\documentclass{article}
\usepackage{amsmath, amsthm, epsfig,amssymb, multicol}
\usepackage[active]{srcltx}
\newcommand{\F}{\mathcal{F}}

\newcommand{\B}{\mathcal{B}}

\newcommand{\E}{{\rm E}}
\renewcommand{\Pr}{{\rm Pr}}
\newtheorem{lemma}{Lemma}
\newtheorem{theorem}{Theorem}
\newtheorem{corollary}{Corollary}
\newtheorem{definition}{Definition}

\newcommand{\st}{\colon\,}

\newcommand{\ceil}[1]{\lceil #1 \rceil}
\def\tsupn{2^{[n]}}

\title{Boolean algebras and Lubell functions}
\author{
Travis Johnston \thanks{University of South Carolina, Columbia, SC 29208,
({\tt johnstjt@mailbox.sc.edu}).} \and 
Linyuan Lu
\thanks{University of South Carolina, Columbia, SC 29208,
({\tt lu@math.sc.edu}). This author was supported in part by NSF
grant  DMS 1000475. }
  \and Kevin G. Milans
\thanks{West Virginia University, Morgantown, WV 26506,
({\tt milans@math.wvu.edu}). } }

\begin{document}
\maketitle

\begin{abstract}
Let $2^{[n]}$ denote the power set of $[n]:=\{1,2,\ldots, n\}$. 
A collection $\B\subset 2^{[n]}$ forms 
a $d$-dimensional {\em Boolean algebra} if there exist pairwise disjoint sets $X_0, X_1,\ldots, X_d \subseteq [n]$,
all non-empty with perhaps the exception of $X_0$,
so that $\B=\left\{X_0\cup \bigcup_{i\in I} X_i\colon I\subseteq [d]\right\}$. 
Let $b(n,d)$  be the maximum cardinality of a family $\F\subset 2^X$
that does not contain a $d$-dimensional Boolean algebra.
Gunderson, R\"odl, and Sidorenko proved that $b(n,d) \leq c_d n^{-1/2^d} \cdot 2^n$ where $c_d= 10^d 2^{-2^{1-d}}d^{d-2^{-d}}$.

In this paper, we use the Lubell function as a new measurement for large families instead of cardinality.
The Lubell value of a family of sets $\F$ with $\F\subseteq \tsupn$ is defined by $h_n(\F):=\sum_{F\in \F}1/{{n\choose |F|}}$.
We prove the following Tur\'an type theorem.
If $\F\subseteq 2^{[n]}$ contains no $d$-dimensional Boolean algebra, then $h_n(\F)\leq 2(n+1)^{1-2^{1-d}}$ for sufficiently large $n$.
This results implies $b(n,d) \leq C n^{-1/2^d} \cdot 2^n$, where $C$ is an absolute constant independent of $n$ and $d$.
As a consequence, we improve several Ramsey-type bounds on Boolean algebras.
We also prove a canonical Ramsey theorem for Boolean algebras.
\end{abstract}

\section{History}
Given a ground set $[n]:=\{1,2,\ldots, n\}$, let $2^{[n]}$ denote the power set of $[n]$.
\begin{definition}
A collection $\B\subseteq 2^{[n]}$ forms a $d$-dimensional {\em Boolean algebra} if there exist pairwise disjoint sets
$X_0, X_1, \ldots,X_d \subseteq [n]$, all non-empty with perhaps the exception of $X_0$, so that 
\[\B=\left\{X_0\cup \bigcup_{i\in I} X_i\colon I\subseteq [d]\right\}.\]  
We view all $d$-dimensional Boolean algebras as copies of a single structure $\B_d$.  Thus, a $d$-dimensional Boolean algebra \emph{forms a copy of $\B_d$}, and a family $\F\subseteq \tsupn$ is \emph{$\B_d$-free} if it does not contain a copy of $\B_d$.
\end{definition}
The starting point of this paper is to explore the question of how large a family of sets can be if it does not contain a $d$-dimensional Boolean algebra.
The first result in this area is due to Sperner.
The simplest example of a non-trivial Boolean algebra, $\B_1$, is a pair of sets, one properly contained in the other.
Sperner's theorem can be restated as follows.
If $\mathcal{F}\subseteq 2^{[n]}$ is $\B_1$-free, then $|\mathcal{F}|\leq \binom{n}{\lfloor \frac{n}{2}\rfloor}$.
Erd\H{o}s and Kleitman~\cite{EK} considered the problem of determining the maximum size of a $\B_2$-free family in $\tsupn$.
General extremal problems on Boolean algebras of sets were most recently studied by Gunderson, R\"odl, and Sidorenko in~\cite{GRS}.

Given an $n$-element set $X$ and a positive integer $d$, define $b(n,d)$ to be the maximum cardinality of a $\B_d$-free family contained in $\tsupn$.  In \cite{GRS}, the following bounds on $b(n,d)$ are proved:
\begin{equation}
  \label{eq:1}
  n^{-\frac{(1+o(1))d}{2^{d+1}-2}}\cdot 2^n \leq b(n,d) \leq 10^d 2^{-2^{1-d}}d^{d-2^{-d}} n^{-1/2^d} \cdot 2^n.
\end{equation}
In the lower bound, the $o(1)$ term represents a function that tends to $0$ as $n$ grows for each fixed $d$.  The lower bound is obtained by considering the affine cubes of integers.

\begin{definition}
A set $H$ of integers is called a $d$-dimensional affine cube or an affine $d$-cube if there exist $d+1$ integers $x_0\geq 0$, and $x_1,\ldots, x_d\geq 1$, such that
\[H=\left\{ x_0+ \sum_{i\in I}x_i\colon I\subseteq [d] \right\}.\]
A set of non-negative integers is \emph{$B_d$-free} if it contains no affine $d$-cube.
\end{definition}

Let $b'(n,d)$ be the maximum size of a $B_d$-free subset of $\{0,\ldots,n\}$.  In \cite{GR}, for sufficiently large $n$, the following bounds on $b'(n,d)$ were proved:
\begin{equation}
  \label{eq:2}
 n^{1-\frac{(1+o(1))d}{2^{d+1}-2}} \leq b'(n,d) \leq  2 (n+1)^{1-2^{1-d}}.
\end{equation}
If $F\subseteq \{0,\ldots, n\}$ and $\F = \{A\in \tsupn\st |A|\in F\}$, then $F$ contains an affine $d$-cube if and only if $\F$ contains a $d$-dimensional Boolean algebra.  Hence, constructions that yield lower bounds on $b'(n,d)$ also yield lower bounds on $b(n,d)$.  Similarly, upper bounds on $b(d,n)$ translate to upper bounds on $b'(d,n)$.  The connection between large $\B_d$-free families in $\tsupn$ and large $B_d$-free families in $\{0,\ldots,n\}$ is greatly simplified when use the Lubell function to measure set families contained in $\tsupn$.

\begin{definition}
Given a family $\F\subseteq 2^{[n]}$, we define the \emph{Lubell function $h_{n}(\F)$} as follows:
\[h_{n}(\F):=\sum_{F\in \F}\frac{1}{\binom{n}{|F|}}.\]
\end{definition}
With this definition in mind, we see that
\begin{equation}\label{eq:b'nd}
b'(n,d) \leq \max \{h_{n}(\F): \F \text{ is } \B_{d}\text{-free}\}
\end{equation}
The Lubell function has been used in the study of extremal families of sets forbidding given subposets (see \cite{diamondfree}).
The Lubell function also has a convenient probabilistic interpretation.
Suppose that $\mathcal{C}$ is a random full-chain in $2^{[n]}$, i.e. $\mathcal{C}=\{\emptyset, \{i_1\}, \{i_1,i_2\},...,[n]\}$.
Let $X$ be the random variable $X=|\mathcal{C}\cap \mathcal{F}|$.
Then we have that $E(X)=h_n(\F)$.

The proof of the upper bound on $b(n,d)$ in inequality~\eqref{eq:1} relies on the Tur\'an density of the $d$-uniform $d$-partite hypergraph in which each part has size $2$.  This, in turn, results in a large multiplicative factor in inequality~\eqref{eq:1} that is asymptotic to $(10d)^d$.  In this paper, using the Lubell function, we improve the upper bound.

\begin{theorem}\label{t:bnd}
There is a positive constant $C$, independent of $d$, such that for every $d$ and all sufficiently large $n$,
the following is true.
\begin{equation}
  \label{eq:bnd}
  b(n,d)\leq C  n^{-1/2^d} \cdot 2^n.
\end{equation}
\end{theorem}

The following theorem uses the Lubell function as the measurement of large families; it implies Theorem \ref{t:bnd}.  The proof may be viewed as an extension of Szemer\'edi's cube lemma (\cite{SCL}; see also problem 14.12 in \cite{LCP}) for $B_d$-free subsets of $\{0,\ldots,n\}$ to $\B_d$-free families contained in $\tsupn$. 
\begin{theorem}\label{turan}
For $d\geq 1$, define $\alpha_d(n)$ recursively as follows.
Let $\alpha_1(n):=1$ and $\alpha_d(n)=\frac{1}{2}+ \sqrt{2n \alpha_{d-1}(n)+\frac{1}{4}}$ for $d\geq 2$.
For $n\geq d\geq 1$ if a family $\F\subseteq 2^{[n]}$ satisfies $h_n(\F)> \alpha_d(n)$, then $\F$ contains a $d$-dimensional Boolean algebra.
\end{theorem}

Note that the sequence  $\{\alpha_d(n)\}_{d\geq 1}$ satisfies
\begin{equation}
  \label{eq:recursive}
  {\alpha_{d+1}(n) \choose 2} = n \alpha_d(n) \quad
\mbox{ for } d\geq 1. 
\end{equation}
We have the following bounds for $n\ge 1$ (see Appendix for the proof).
\begin{equation}
  \label{eq:trivial}
  (2n)^{1-2^{1-d}}\leq \alpha_d(n) \leq (4n)^{1-2^{1-d}}.   
\end{equation}
when $n\geq d\geq 1$.

The function $\alpha_d(n)$ is used in \cite{GR} implicitly.
Note that for any fixed $d\geq 2$, $\alpha_d(n)$ is an increasing function of $n$.
We have $\alpha_1(n)=1$, $\alpha_2(n)=\frac{1}{2}+ \sqrt{2n +\frac{1}{4}}$.
For $d\geq 3$, it was implicitly shown in \cite{GR} that 
\[\alpha_d(n) \leq 2^{1-2^{1-d}}(\sqrt{n+1}+1)^{2-2^{2-d}} \quad \mbox{ for } n+1\geq 2^{d2^{d-1}/(2^{d-1}-1)}\]
and 
\[\alpha_d(n) \leq 2 (n+1)^{1-2^{1-d}} \quad \mbox{ for } n+1\geq (2^d-2/\ln 2)^2.\]

The following is a corollary which can be viewed as the generalization of inequality \eqref{eq:2} and \eqref{eq:b'nd}.

\begin{corollary}
For $d\geq 3$ and $n\geq  (2^d-2/\ln 2)^2$,
every family $\F\subseteq \tsupn$ containing no $d$-dimensional Boolean algebra
satisfies  $h_n(\F)\leq  2 (n+1)^{1-2^{1-d}}$.
\end{corollary}

The rest of the paper is organized as follows. In section 2, we prove
Theorem  \ref{t:bnd} and Theorem \ref{turan}. In section 3, we prove several
Ramsey-type results.

\section{Proofs of Theorems \ref{t:bnd} and \ref{turan}}

{\bf Proof of Theorem \ref{turan}:} 

The proof is by induction on $d$.
For the initial case $d=1$, we have $h_n(\F)> \alpha_1(n)=1$.
Let $X$ be the number of sets in both $\F$ and a random full chain. 
Then $\E(X)=h_n(\F)>1$. 
There is an instance of $X$ satisfying $X\geq 2$.
Let $A$ and $B$ be two sets in both $\F$ and a full chain.
Clearly, the pair $\{A, B\}$ forms a copy of $\B_1$.

Assume that the statement is true for $d$.
For $d+1$, suppose $\F\subseteq \tsupn$ satisfies $h(\F)> \alpha_{d+1}(n)$.
Let $X$ be the number of sets in both $\F$ and a random full chain.
By the convex inequality, we have
\begin{align*}
	\E{X\choose 2} &\geq \binom{\E X}{2} \\
                       &> \binom{\alpha_{d+1}(n)}{2} \\
                       &= n \alpha_d(n).
\end{align*}
For each subset $S$ of $[n]$, let $\F_S = \{A\in \F\st\mbox{$A\cap S = \emptyset$ and $A\cup S\in\F$}\}$.  We show that for some non-empty set $S$, the Lubell function of $\F_S$ in $2^{[n]\setminus S}$ exceeds $\alpha_d(n-|S|)$.  It follows by induction that $\F_S$ contains a copy of $\B_d$ generated by some sets $S_0,S_1,\ldots,S_d$, and with $S$ these sets generate a copy of $\B_{d+1}$ in $\F$.  Let $Z = \{(A,B)\in \F\times\F\st \mbox{$A\subsetneq B$}\}$.  For each $(A,B)\in Z$, the probability that a random full-chain in $\tsupn$ contains both $A$ and $B$ is $1/\binom{n}{|A|,|B|-|A|,n-|B|}$.  We compute
\begin{align*}
	\E{X\choose 2} &= \sum_{(A,B)\in Z}\frac{1}{\binom{n}{|A|,|B|-|A|,n-|B|}}\\
		       &= \sum_{\emptyset \subsetneq S \subseteq [n]} \sum_{A\in \F_S} \frac{1}{\binom{n}{|A|,|S|,n-|A|-|S|}}\\
		       &= \sum_{\emptyset \subsetneq S \subseteq [n]} \frac{1}{\binom{n}{|S|}} \sum_{A\in \F_S} \frac{1}{\binom{n-|S|}{|A|}}\\
		       &= \sum_{\emptyset \subsetneq S \subseteq [n]} \frac{1}{\binom{n}{|S|}} h_{n-|S|}(\F_S)\\
		       &= \sum_{k=1}^n \frac{1}{\binom{n}{k}} \sum_{S \in \binom{[n]}{k}}  h_{n-k}(\F_S).\\
\end{align*}
Since $\E{\binom{X}{2}} > n\alpha_d(n)$, it follows that $\frac{1}{\binom{n}{k}} \sum_{S \in \binom{[n]}{k}}  h_{n-k}(\F_S) > \alpha_d(n)$ holds for some $k$.  In turn, $h_{n-k}(\F_S) > \alpha_d(n) \ge \alpha_d(n-k)$ for some set $S$ of size $k$.
\hfill $\square$

Before proving Theorem \ref{t:bnd}, we need bounds on ratios of binomial coefficients.
\begin{lemma}\label{l:bin}
If $k \le n$, then $\binom{2n}{k}/\binom{2n}{n} \le e^{-\frac{2}{n}\binom{n-k}{2}}$.
\end{lemma}
\begin{proof}
Note that $\binom{2n}{k}/\binom{2n}{n} = \frac{n!\cdot n!}{k!(2n-k)!} = \prod_{j=0}^{n-k-1} \frac{n-j}{n+j+1} \le \prod_{j=0}^{n-k-1} \frac{n-j}{n+j}$.  Next, we apply the inequality $(1-x)/(1+x)\le e^{-2x}$ for $x\ge 0$ to find $\binom{2n}{k}/\binom{2n}{n} \le e^{-\frac{2}{n}\sum_{j=0}^{n-k-1} j} = e^{-\frac{2}{n}\binom{n-k}{2}}$.
\end{proof}

{\bf Proof of Theorem \ref{t:bnd}: }
Let $\F\subseteq 2^{[n]}$ be a $\B_d$-free family.  For $0\le a \le b\le n$, let $\F(a,b) = \{A\in\F\st a\le |A|\le b\}$.  For two sets $A$ and $B$ with $A\subseteq B$, the \emph{interval} $[A,B]$ is the set $\{X\in \tsupn\st A\subseteq X\subseteq B\}$.  Let $Z = \{(A,B)\st \mbox{$A\subseteq B$, $|A| = a$, and $|B| = b$}\}$.  Since $\F$ is $\B_d$-free and $[A,B]$ is a copy of the $(b-a)$-dimensional Boolean algebra, Theorem~\ref{turan} implies that $h_{b-a}(\F\cap[A,B]) \le \alpha_d(b-a)$ for each $(A,B)\in Z$.  Since a random chain is equally likely to intersect levels $a$ and $b$ at all pairs in $Z$, it follows that $h_n(\F(a,b))$ is the average, over all $(A,B) \in Z$, of $h_{b-a}(\F\cap[A,B])$.  Therefore $h_n(\F(a,b)) \le \alpha_d(b-a)$.

We may assume without loss of generality that $n$ is an even integer $2m$, and let $\ell = \ceil{\sqrt{m}}$.  We first bound the number of sets in $\F$ whose size is at most $m$; to do this, we partition $\{A\in\F\st |A|\le m\}$ into subsets of the form $\F(a,b)$ where $b-a$ is at most $\ell$.   Let $t$ be the largest integer such that $m - t\ell - 1 \ge 0$.  We define $x_0,\ldots,x_{t+1}$ by setting $x_0 = m$, $x_j = m - j\ell - 1$ for $1\le j\le t$, and $x_{t+1} = -1$.  For $0 \le j \le t$, we define $\F_j = \F(x_{j+1} + 1, x_j)$, and note that $x_j - (x_{j+1} + 1) \le \ell$ for all $j$.  Hence $h_n(\F_j) \le \alpha_d(\ell)$ for all $j$.  Since $h_n(\F_j) \ge |\F_j|/\binom{2m}{x_j}$, it follows that $|\F_j| \le \alpha_d(\ell)\binom{2m}{x_j}$.

We compute
\begin{align*}
\sum_{j=0}^t |\F_j| &\le \alpha_d(\ell)\sum_{j=0}^t \binom{2m}{x_j}\\
&\le \alpha_d(\ell) \binom{2m}{m} \sum_{j=0}^t e^{-\frac{2}{m}\binom{m-x_j}{2}}\\
&\le \alpha_d(\ell) \binom{2m}{m} \sum_{j=0}^t e^{-\frac{1}{m}(j\ell)^2}\\
&\le \alpha_d(\ell) \binom{2m}{m} \sum_{j\ge0} e^{-\frac{\ell^2}{m}j}\\
&\le \alpha_d(\ell) \binom{2m}{m} \frac{1}{1-e^{-\ell^2/m}},
\end{align*}
where we have applied Lemma~\ref{l:bin}.  Since $\ell\ge \sqrt{m}$, the series is bounded by the absolute constant $1/(1-e^{-1})$.  Using that $\binom{2m}{m} \le \frac{\sqrt{2}e}{2\pi} \frac{1}{\sqrt{m}} 2^{2m}$ for all $m$ and applying our bound $\alpha_d(\ell) \le (4\ell)^{1-2^{1-d}} \le (4(\sqrt{m} + 1))^{1-2^{1-d}} \le 8(\sqrt{m})^{1-2^{1-d}}$ yields 
\[ \sum_{j=0}^t |\F_j| \le \frac{8\sqrt{2}e^2}{2\pi(e-1)} \cdot m^{-1/2^d} \cdot 2^{2m}. \]
Doubling this, we have that $|\F| \le \frac{8\sqrt{2}e^2}{\pi(e-1)} \cdot m^{-1/2^d} \cdot 2^{2m}$, and substituting $m=n/2$ gives $|\F| \le \frac{16e^2}{\pi(e-1)} \cdot n^{-1/2^d} \cdot 2^n < 22n^{-1/2^d}2^n$.
\hfill $\square$

We note that our constant $22$ can be reduced by sharpening the analysis in the proof of Theorem~\ref{t:bnd} in several places; we make no attempt to further reduce the constant.

\section{Ramsey-type results}

\subsection{Multi-color Ramsey results}
Given positive integers $n$ and $d$, define $r(d,n)$ to be the largest integer
$r$ so that every $r$-coloring of $\tsupn$ contains a monochromatic copy
of $\B_d$.  Gunderson, R\"odl, and Sidorenko~\cite{GRS} proved
for $d>2$,
\begin{equation}
  \label{eq:rdn}
  c n^{1/2^d} \leq r(d,n) \leq n^{\frac{d}{2^d-1}(1+o(1))}.
\end{equation}

Using Theorem \ref{turan}, we improve the lower bound.

\begin{theorem} \label{t:rdn}
  For $d>2$, we have \[r(d,n)\geq \frac{1}{2} (n+1)^{2/2^d}.\]
\end{theorem}

{\bf Proof of Theorem~\ref{t:rdn}:} Let $r=\frac{1}{2} (n+1)^{2/2^d}$. For every $r$-coloring of $\tsupn$
and $1\leq i \leq r$, let $\F_i$ be the family of sets in color $i$.
By linearity, we have
\[\sum_{i=1}^r h_n(\F_i)=h_n(\tsupn)=n+1.\]
By the pigeon hole principle, there is a color $i$ with $h_n(\F_i)\geq \frac{n+1}{r}=2(n+1)^{1-2^{1-d}}$.
For all $r,d\geq 2$, we have $n+1 \geq (2^d-2/\ln 2)^2.$ 
Thus,
\[h_n(\F_i)\geq \frac{n+1}{r}= 2 (n+1)^{1-2^{1-d}}>\alpha_{d}(n).\]
By Theorem \ref{turan}, $\F_i$ contains a copy of $\B_d$.
\hfill $\square$

For positive integers $t_1, t_2, \ldots, t_r$, let $R(\B_{t_1},\ldots, \B_{t_r})$ be the least integer $N$
such that for any $n\geq N$ and any $r$-coloring of $\tsupn$ there exists an $i$ such that $\B_n$ contains a monochromatic copy of $\B_{t_i}$ in color $i$.  In this language, Theorem~\ref{t:rdn} states that 
\[R(\underbrace{\B_t,\ldots, \B_t}_r) \le (2r)^{2^{t-1}}-1.\]  Next, we establish an exact result for $R(\B_s, \B_1)$.  
Our lower bound on $R(\B_s, \B_1)$ requires a numerical result.  A sequence of positive integers is \emph{complete} if every positive integer is the sum of a subsequence.  In 1961, Brown~\cite{Brown} showed that a non-decreasing sequence $x_1,x_2,\ldots$ of positive integers with $x_1=1$ is complete if and only if $\sum_{i=1}^k x_i \le 1 + x_{k+1}$ for each $k$.  We adapt Brown's argument to obtain a sufficient condition for a finite variant; we include the proof for completeness.

\begin{lemma}[Brown~\cite{Brown}]\label{l:sums}
Let $x_1,\ldots,x_s$ be a list of positive integers with sum at most $2s-1$.  For each $k$ with $0\le k\le s$, there is a sublist with sum $k$.
\end{lemma}
\begin{proof}
We use induction on $s$.  Since the empty list of numbers has sum $0$ which is larger than $2\cdot 0 - 1$, the lemma holds vacuously when $s=0$.  For $s\ge 1$, index the integers so that $1\le x_1 \le \cdots \le x_s$.  If $x_s = 1$, then $x_j = 1$ for each $j$ and the lemma holds.  Otherwise, $x_s \ge 2$ and $x_1,\ldots,x_{s-1}$ has sum at most $2(s-1)-1$.  By induction, for each $k$ with $0\le k\le s-1$, some sublist of $x_1,\ldots,x_{s-1}$ has sum $k$.  Note that $x_s \le s$, or else $x_s \ge s+1$ and $x_j \ge 1$ for $1\le j\le s-1$ would contradict that the list $x_1,\ldots,x_s$ has sum at most $2s-1$.  Since $s-x_s$ is in the range $\{0,\ldots,s-1\}$, we obtain a sublist with sum $s$ by adding $x_s$ to a sublist of $x_1,\ldots,x_{s-1}$ with sum $s-x_s$.
\end{proof}

\begin{theorem}
 For all $s\geq 1$, we have $R(\B_s,\B_1)=2s$.
\end{theorem}
{\bf Proof:} First we show $R(\B_s,\B_1)\leq 2s$.
Let $n = 2s$, let $c$ be a red-blue coloring of $\tsupn$, and suppose for a contradiction that $c$ contains neither a red copy of $\B_s$ nor a blue copy of $\B_1$.  We claim that every blue set has size $s$.  If $A$ is blue, then all points in the up-set of $A$ and all points in the down-set of $A$ are red, or else the coloring has a blue copy of $\B_1$.  If $|A| < s$, then the up-set of $A$ contains red copies of $\B_s$, and if $|A| > s$, then the down-set of $A$ contains red copies of $\B_s$.  Therefore $|A| = s$ as claimed.  Consider the copy of $\B_s$ generated via setting $X_0 = \emptyset$, $X_j = \{j\}$ for $1\le j \le s-1$, and $X_s = \{s,s+1,\ldots,2s\}$.  None of the sets in this copy of $\B_s$ have size $s$, and therefore this is a red copy of $\B_s$, a contradiction.

Now we show that $R(\B_s,\B_1) > 2s-1$. Let $n=2s-1$.  We construct a $2$-coloring of $\tsupn$ that contains no red copy of $\B_s$ and no blue copy of $\B_1$ as follows.
Color all sets of size $s$ blue and all other sets red.  The blue sets form an antichain, so the coloring avoids blue copies of $\B_1$.  It suffices to show that there is no red copy of $\B_s$.  Suppose for a contradiction that a red copy of $\B_s$ is generated by sets $X_0,X_1,\ldots,X_s$, and let $x_j = |X_j|$.  Since $X_1,\ldots,X_s$ are disjoint in $[2s-1]$ and non-empty, it follows that $x_1,\ldots,x_s$ is a list of positive integers with sum at most $2s-1$.  Since $0\le x_0 \le s-1$, we apply Lemma~\ref{l:sums} to obtain $I\subseteq [s]$ such that $\sum_{i\in I} x_i = s - x_0$.  It follows that $X_0 \cup \bigcup_{i\in I} X_i$ has size $s$, contradicting that the coloring contains a red copy of $\B_s$.
\hfill $\square$

In 1950, Erd\H{o}s and Rado~\cite{ERCanon} proved the Canonical Ramsey Theorem, which lists structures that arise in every edge-coloring of the complete graph on countably many vertices.  It states that each edge-coloring of the complete graph on the natural numbers contains an infinite subgraph $H$ such that either all edges in $H$ have the same color, or the edges in $H$ have distinct colors, or the edges in $H$ are colored lexicographically by their minimum or maximum endpoint.  By a standard compactness argument, the Canonical Ramsey Theorem implies a finite version, stating that for each $r$, there is a sufficiently large $n$ such that every edge-coloring of the complete graph on vertex set $[n]$ contains a subgraph on $r$ vertices that is colored as in the infinitary version.  

\newcommand{\CR}{\mathrm{CR}}
An analogous result holds for colorings of $\tsupn$.  A coloring of $\tsupn$ is \emph{rainbow} if all sets receive distinct colors.  Let $\CR(r,s)$ be the minimum $n$ such that every coloring of $\tsupn$ contains a rainbow copy of $\B_r$ or a monochromatic copy of $\B_s$.  Although it is not immediately obvious that $\CR(r,s)$ is finite, our next theorem provides an upper bound.

\begin{theorem} \label{canramsey}
$\CR(r,s)\le r 2^{(2r+1)2^{s-1} - 2}$ for positive $r$ and $s$.
\end{theorem}
\begin{proof}
Set $t = 2^{(2r+1)2^{s-1} - 2}$ and $n=tr$, and consider a coloring of $\tsupn$ that does not contain a monochromatic copy of $\B_s$.  We obtain a rainbow copy of $\B_r$ with the probabilistic method.  Partition the ground set $[n]$ into $r$ sets $U_1, \ldots, U_r$ each of size $t$.  Independently for each $i$ in $[r]$, choose a subset $X_i$ from $U_i$ so that sets are chosen proportionally to their Lubell mass in the Boolean algebra on $U_i$.  That is, each $k$-set in $U_i$ has probability $\frac{1}{t+1} \binom{t}{k}^{-1}$ of being selected for $X_i$.  For each pair $\{I,J\}$ with $I, J\subseteq[r]$, let $A_{I,J}$ be the event that both $\bigcup_{i\in I} X_i$ and $\bigcup_{j\in J} X_j$ receive the same color.  

We obtain an upper bound on the probability that $A_{I,J}$ occurs.  Since $I$ and $J$ are distinct sets, we may assume without loss of generality that there exists $m\in I-J$.  Fix the selection of all sets $X_1,\ldots,X_r$ except $X_m$.  This determines the color $c$ of $\bigcup_{j\in J} X_j$, and the probability that $A_{I,J}$ occurs is at most the probability that $\bigcup_{i\in I} X_i$ has color $c$.  Let $L$ be the $t$-dimensional Boolean sublattice with ground set $U_m$, and color $B\in L$ with the same color as $B\cup \bigcup_{i\in I-\{m\}} X_i$.  Let $\F$ be the elements in $L$ with color $c$.  Since $\F$ does not contain a monochromatic copy of $\B_s$, Theorem~\ref{turan} implies that $h_t(\F) \le (4t)^{1-2^{1-s}}$.  Since 
\begin{align*}
h_t(\F) &= \sum_{B\in \F} \binom{t}{|B|}^{-1} \\
	&= (t+1) \sum_{B\in \F} \Pr[X_m = B] \\
	&\ge (t+1) \Pr[A_{I,J}],
\end{align*}
we have that $\Pr[A_{I,J}] \le (4t)^{1-2^{1-s}}/(t+1) < 4/(4t)^{2^{1-s}}$.  Using the union bound, we have that the probability that at least one of the events $A_{I,J}$ occurs is less than $\binom{2^r}{2}\cdot 4/(4t)^{2^{1-s}}$, which is at most $1$.  It follows that for some selection of the sets $X_1,\ldots X_r$, none of the events $A_{I,J}$ occur.  These sets generate a rainbow copy of $\B_r$.
\end{proof}

Note that Equation~\eqref{eq:rdn} implies that if $k > n^{\frac{s}{2^s - 1}(1+o(1))}$, then there is a $k$-coloring of $\tsupn$ that does not contain a monochromatic copy of $\B_s$.  Of course, with $k=2^r-1$, there is also no rainbow copy of $\B_r$.  It follows that $2^r-1 > n^{\frac{s}{2^s - 1}(1+o(1))}$ implies that $\CR(r,s)> n$, and hence $\CR(r,s) \ge 2^{\frac{r(2^s-1)}{s}(1-o(1))}$ where the $o(1)$ term tends to $0$ as $r$ increases.

\section{Appendix}
{\bf Proof of inequality \eqref{eq:trivial}:}
We first prove the lower bound by induction on $d$.
For $d=2$, we have $$(2n)^{1-2^{1-d}}=\sqrt{2n}\leq \frac{1}{2}+\sqrt{2n+\frac{1}{4}}
= \alpha_d(n).$$

Assume that $(2n)^{1-2^{1-d}}\leq \alpha_d(n)$ holds for $d$.
For $d+1$, we have
$$n \alpha_d(n) ={\alpha_{d+1}(n) \choose 2}< \frac{\alpha_{d+1}(n)^2}{2}.$$
Thus
$$\alpha_{d+1}(n)> \sqrt{2n \alpha_d(n)}\geq 
\sqrt{2n (2n)^{1-2^{1-d}}}= (2n)^{1-2^{-d}}.$$
The proof of induction is finished.

Now we prove the upper bound by induction on $d$.
For $d=2$, we have 
$$\alpha_d(n)=\frac{1}{2}+\sqrt{2n+\frac{1}{4}} \leq \sqrt{4n}= (4n)^{1-2^{1-d}}.$$
Assume that $\alpha_d(n)\leq (4n)^{1-2^{1-d}}$ holds for $d$.
Since $\alpha_{d+1}(n)\geq (2n)^{1-2^{-d}}\geq 2$,
for $d+1$, we have
$$n \alpha_d(n) ={\alpha_{d+1}(n) \choose 2}> \frac{\alpha_{d+1}(n)^2}{4}.$$
Thus
$$\alpha_{d+1}(n)< \sqrt{4n \alpha_d(n)}\leq
\sqrt{4n (4n)^{1-2^{1-d}}}= (4n)^{1-2^{-d}}.$$
The proof of induction is finished.
\hfill $\square$

\end{document}